\documentclass[11pt]{amsart}

\usepackage{amsfonts, amssymb, amscd, amsmath, enumerate, verbatim}

\newtheorem{theorem}{Theorem}

\newtheorem{proposition}[theorem]{Proposition}
\newtheorem{corollary}[theorem]{Corollary}
\newtheorem{example}[theorem]{Example}

\newtheorem{remark}[theorem]{Remark}

\begin{document}

\title{Classification of Codazzi and note on minimal hypersurfaces in $Nil^{4}$}

\author{Noura Djellali${^1}$, Abdelbasset Hasni${^1}$, \\ Ahmed Mohammed Cherif${^{1,*}}$ \and Mohamed Belkhelfa${^1}$}

\footnote{ Department of  Mathematics, University Mustapha Stambouli, Mascara 29000, Algeria.\\
${^*}${\it Corresponding author, e-mail: a.mohammedcherif@univ-mascara.dz}}

\date{}

\maketitle


\begin{abstract}
In this paper, we give a classification of Codazzi hypersurfaces in a Lie group $(Nil^{4},\widetilde g)$.
We also give a characterization  of a class of minimal hypersurfaces in $(Nil^{4},\widetilde g)$ with an example of a minimal surface in this class.
\\
\textit{ keywords:} Codazzi hypersurfaces, minimal hypersurfaces.\\
\textit{ Mathematics Subject Classification 2020:} 53C42, 53C35.
\end{abstract}

\section{Introduction}
Let $( M^m ,g)$ be a Riemannian manifold, $ \nabla$, $ R$, $ S $ and $ \tau $ denote the Levi-Civita connection, the Riemannian curvature, the Ricci curvature, and the scalar curvature of $(M,g)$, respectively. Thus
\begin{equation}\label{eq1.1}
 R(X,Y)Z=[\nabla_{X},\nabla_{Y}]Z-\nabla_{[X,Y]}Z,
\end{equation}
\begin{eqnarray}\label{eq1.2}
  S(X,Y)&=&\sum_{i=1}^{m} g(R(X,e_{i})e_{i},Y),
\end{eqnarray}
\begin{equation}\label{eq1.3}
\tau=\sum_{i,j=1}^{m}g(R(e_{i},e_{j})e_{j},e_{i}),
\end{equation}
where $\{e_i\}_{1\leq i\leq m}$ is an orthonormal frame on $(M,g)$, and $X,Y,Z\in\mathfrak{X}(M^m)$. A symmetric $(0,2)$-tensor field $T$ on $(M^m,g)$ is said to be a Codazzi tensor if it satisfies the Codazzi equation
\begin{equation}\label{eq1.4}
 (\nabla_{X}T)(Y,Z)= (\nabla_{Y}T)(X,Z),\quad\forall X,Y,Z \in\mathfrak{X}(M^m),
\end{equation}
(see \cite{ON,W}). \\
Let $(N^{n},g)$ be a hypersurface in a Riemannian manifold $(M^{n+1},\widetilde g)$, where $g$ is the induced Riemannian metric by $\widetilde g$. We denote by $\nabla$ (resp. $\widetilde{\nabla}$) the  Levi-Civita connection of $(N^n,g)$ (resp. of $(M^{n+1},\widetilde g)$),
$R$ (resp. $\widetilde{R}$) the Riemannian curvature of $(N^{n},g)$ (resp. of $(M^{n+1},\widetilde g)$),
$B(\cdot,\cdot)=h(\cdot,\cdot)\xi$ the second fundamental form of the hypersurface  $(N^n,g)$, $A_\xi$ the shape operator with respect to the unit normal vector field $\xi$,  $H=(1/n)\operatorname{trace}_gB$ the mean curvature of $(N^n,g)$
(see \cite{ON}). Under the notation above, we have
\begin{equation}\label{eq1.5}
 \widetilde\nabla_{X}Y= \nabla_{X}Y+h(X,Y)\xi,
\end{equation}
\begin{equation}\label{eq1.6}
A_{\xi}X=-\widetilde{\nabla}_{X}\xi,\quad\forall X,Y \in\mathfrak{X}(N^n),
\end{equation}
Note that, the components of the second fundamental form $B$ are given by
\begin{equation}\label{eq1.7}
h(X,Y)= g(A_\xi X,Y)=-g(\widetilde{\nabla}_{X}\xi,Y),\quad\forall X,Y \in\mathfrak{X}(N^n).
\end{equation}
The equations of Gauss and Codazzi are given respectively by
\begin{equation}\label{eq1.8}
    \widetilde{g}(\widetilde{R}(X,Y)Z,W)=g(R(X,Y)Z,W)+h(X,W)h(Y,Z)-h(X,Z)h(Y,W),
\end{equation}
\begin{equation}\label{eq1.9}
 \widetilde{g}(\widetilde{R}(X,Y)Z,\xi) =(\nabla_Y h)(X,Z)-(\nabla_X h)(Y,Z),
\end{equation}
where $(\nabla_X h)(Y,Z)=X(h(Y,Z))-h(\nabla_XY,Z)-h(Y,\nabla_XZ)$ called the cubic form, and $X,Y,Z,W\in \mathfrak{X}(N^n)$. The hypersurface $(N^n,g)$ is said to be parallel if the cubic form vanishes identically, i.e., $\nabla h=0$. A special case of parallel hypersurfaces are totally geodesic hypersurfaces, for which the second fundamental form $B=0$. The hypersurface $(N^n,g)$ is called Codazzi (resp. minimal) if the symmetric $(0,2)$-tensor field $h$ is a Codazzi tensor (resp. if $H=0$) (see \cite{ON}).\\
The $4$-dimensional lie group $Nil^{4}$ is a well known nilpotent lie group. It is also one of the $4$-dimensional thurston model geometries,  \cite{Filip}.
As the $Nil^{4}$ space are well know and one of its left invariant Riemannian metric is well known and also used in many research works, we will begin with an explicit calculus of this metric and its geometric proprieties, (that one of the autors has do a part of it, in \cite{H}, based on the definition of $Nil^{4}$ and its most used left invariant metric), only to let the reader follow us easily.

The nilpotent Lie group $Nil^{4}=\mathbb{R}^3\ltimes_{U}\mathbb{R}$,  where $ U(t)=\exp(tL)$, with
\begin{equation}\label{eq1.10}
 L=\begin{pmatrix}
 0 & 1 & 0\\
 0 & 0 & 1 \\
 0  & 0 & 0
\end{pmatrix},\quad
 \exp(tL)=I_{3}+tL+\frac{t^{2}}{2}L^{2}=\begin{pmatrix}
 1 & t & \frac{t^{2}}{2}\\
 0 & 1 & t \\
 0  & 0 & 1
\end{pmatrix}.
\end{equation}
The semidirect product in $ Nil^{4} $ is given by
\begin{eqnarray}\label{eq1.11}
   (V,t)(V^{\prime},t^{\prime})&=&\nonumber(V+\exp(t L)V^{\prime},t+t^{\prime})\\
                               &=&\nonumber\left(\begin{pmatrix}
 x\\
 y \\
 z
\end{pmatrix}+ \begin{pmatrix}
 1 & t & \frac{t^{2}}{2}\\
 0 & 1 & t \\
 0  & 0 & 1
\end{pmatrix}\begin{pmatrix}
 x^{\prime}\\
 y^{\prime} \\
 z^{\prime}
\end{pmatrix},t+t^{\prime}\right)\\
   &=&\left(\begin{pmatrix}
 x+x^{\prime}+t y^{\prime}+ \frac{t^{2}}{2}z^{\prime}\\
 y+y^{\prime}+tz^{\prime} \\
 z+z^{\prime}
\end{pmatrix},t+t^{\prime}\right),
         \end{eqnarray}
for all
$
 V=\begin{pmatrix}
 x\\
 y \\
 z
\end{pmatrix},
$
$
 V^{\prime}=\begin{pmatrix}
 x^{\prime}\\
 y^{\prime} \\
 z^{\prime}
\end{pmatrix} \in\mathbb{R}^3,$ and $t\in\mathbb{R}$. We have the parameterization
\begin{eqnarray}\label{eq1.12}
  \phi: Nil^{4} &\longrightarrow &\nonumber \mathbb{R}^4 . \\
                           \left(\begin{pmatrix}
 x\\
 y \\
 z
\end{pmatrix},t\right) & \longmapsto & (x,y,z,t)
\end{eqnarray}
Taking the left-invariant frame fields
\begin{eqnarray}\label{eq1.13}
&  &\nonumber e_{1}=\frac{\partial}{\partial x}, \\
&  &\nonumber e_{2}=t\frac{\partial}{\partial x}+\frac{\partial}{\partial y},\\
&  & e_{3}=\frac{t^{2}}{2}\frac{\partial}{\partial x}+t\frac{\partial}{\partial y}+\frac{\partial}{\partial z},\\
&  &\nonumber e_{4}=\frac{\partial}{\partial t}.
\end{eqnarray}
So that, the dual coframe fields are given by
\begin{eqnarray}\label{eq1.14}
&  & \theta_{1}=\nonumber dx-tdy+\frac{t^{2}}{2}dz, \\
&  & \theta_{2}=\nonumber dy-tdz,\\
&  & \theta_{3}=dz,\\
&  & \theta_{4}=\nonumber dt.
 \end{eqnarray}
The matrix of a Riemannian metric $ \widetilde g=\theta_{1}^{2}+\theta_{2}^{2}+\theta_{3}^{2}+\theta_{4}^{2}$ is given by
$$ (\widetilde{g}_{ij})=\begin{pmatrix}
 1 & -t & \frac{t^{2}}{2} & 0\\
 -t & 1+t^{2} &  -t(1+\frac{t^{2}}{2}) & 0 \\
 \frac{t^{2}}{2} & -t(1+\frac{t^{2}}{2})  & 1+t^{2}+\frac{t^{4}}{4} & 0 \\
   0 & 0 & 0& 1
\end{pmatrix},
$$
In this paper, we study some geometric properties of a Riemannian manifold $(Nil^{4},\widetilde g)$.
We also give a characterization  of a minimal hypersurfaces in $(Nil^{4},\widetilde g)$ which have a normal vector field depend only on, last coordinate, $t$.

\section{Geometric properties of $(Nil^{4},\widetilde g)$}

\begin{proposition}\label{prop1}
The non-zero of the Levi-Civita connection $\widetilde{\nabla} $ of $(Nil^{4},\widetilde g)$ are given by
\begin{eqnarray*}
\widetilde\nabla_{e_{1}}e_{2}&=&\frac{1}{2}e_{4} \,, \quad\quad\quad\quad\quad\,\,\,\,\,\, \widetilde\nabla_{e_{1}}e_{4}=-\frac{1}{2}e_{2}\\
\widetilde\nabla_{e_{2}}e_{1}&=&\frac{1}{2}e_{4} \,,   \quad\quad\quad\quad\quad\,\,\,\,\,\, \widetilde\nabla_{e_{2}}e_{3}=\frac{1}{2}e_{4}\\
\widetilde\nabla_{e_{2}}e_{4}&=&-\frac{1}{2}(e_{1}+e_{3}) \,, \quad\quad\,\,   \widetilde\nabla_{e_{3}}e_{2}=\frac{1}{2}e_{4} \\
\widetilde\nabla_{e_{3}}e_{4}&=&-\frac{1}{2}e_{2}  \,,   \quad\quad\quad\quad\quad \widetilde\nabla_{e_{4}}e_{1}=-\frac{1}{2}e_{2}\\
\widetilde\nabla_{e_{4}}e_{2}&=&\frac{1}{2}(e_{1}-e_{3}) \,,    \quad\quad\quad\widetilde\nabla_{e_{4}}e_{3}=\frac{1}{2}e_{2} .
\end{eqnarray*}
\end{proposition}
\begin{proof}
Note that, the non-zero of Christoffel symbols $\widetilde{\Gamma}^{k}_{ij} $ for $i,j,k \in\lbrace 1,2,3,4\rbrace $ are given by
\begin{eqnarray*}
 \widetilde{\Gamma}^{4}_{12}&=&\frac{1}{2} \,,\quad\quad\quad\quad\,\,\,\,\,\,\widetilde{\Gamma}^{4}_{13}=-\frac{t}{2}\\
 \widetilde{\Gamma}^{1}_{14}&=&-\frac{t}{2} \,,\quad \quad\quad\quad\widetilde{\Gamma}^{2}_{14}=-\frac{1}{2} \\
 \widetilde{\Gamma}^{4}_{22}&=&-t \,,\quad\quad\quad\quad\,\,\,\widetilde{\Gamma}^{4}_{23}=\frac{1}{2}+\frac{3t^{2}}{4} \\
\widetilde{\Gamma}^{1}_{24}&=&-\frac{1}{2}+\frac{t^{2}}{4}\,,\quad\quad\widetilde{\Gamma}^{3}_{24}=-\frac{1}{2}\\
\widetilde{\Gamma}^{4}_{33}&=&-t(1+\frac{t^{2}}{2})\,,\,\quad\widetilde{\Gamma}^{2}_{34}=-\frac{1}{2}+\frac{t^{2}}{4} \\
\widetilde{\Gamma}^{3}_{34}&=&\frac{t}{2}\,.
\end{eqnarray*}
Proposition \ref{prop1} follows from (\ref{eq1.13}).
\end{proof}

\begin{corollary}\label{co1}
The non-zero Lie brackets of the basis $ \lbrace e_{i}\rbrace_{1\leq i\leq 4} $ are given by
\begin{eqnarray*}
  [e_{4}, e_{2}] = e_{1} & , &  [e_{4}, e_{3}]=e_{2} .
\end{eqnarray*}
\end{corollary}

\begin{proof}
Follows directly by Proposition \ref{prop1}, with $[e_{i}, e_{j}]=\widetilde\nabla_{e_{i}}e_{j}-\widetilde\nabla_{e_{j}}e_{i}$ for all $i,j=1,2,3,4$.
\end{proof}

\begin{proposition}\label{prop2}
The only non-zero components of Riemannian curvature of $(Nil^{4},\widetilde{g})$ are given by
\begin{eqnarray*}
 \widetilde{g}(\widetilde{R}(e_1,e_2)e_1,e_2) &=& -\frac{1}{4}\,,\quad\quad\,\widetilde{g}(\widetilde{R}(e_1,e_2)e_2,e_3) = \frac{1}{4} \\
 \widetilde{g}(\widetilde{R}(e_1,e_4)e_1,e_4) &=& -\frac{1}{4}\,,\quad\quad\,\widetilde{g}(\widetilde{R}(e_1,e_4)e_3,e_4) = \frac{1}{4} \\
\widetilde{g}(\widetilde{R}(e_2,e_1)e_2,e_3)  &=& -\frac{1}{4}\,,\quad\quad\,\widetilde{g}(\widetilde{R}(e_2,e_3)e_2,e_3) = -\frac{1}{4} \\
\widetilde{g}(\widetilde{R}(e_2,e_4)e_2,e_4)  &=& \frac{1}{2}\,,\quad\quad\quad\widetilde{g}(\widetilde{R}(e_3,e_4)e_3,e_4)  = \frac{3}{4}.
\end{eqnarray*}
\end{proposition}
\begin{proof}
Using the definition of Riemannian curvature (\ref{eq1.1}), the Proposition \ref{prop1}, and the Corollary \ref{co1}.
\end{proof}
According to Proposition \ref{prop2}, we have the following Corollary.
\begin{corollary}
The matrix of Ricci curvature of $(Nil^{4},\widetilde{g})$ is given by
$$(S_{ij})=\begin{pmatrix}
 \frac{1}{2} & 0 & 0 & 0 \\
        0    & 0 & 0 & 0 \\
        0    & 0 & -\frac{1}{2} & 0\\
        0  &   0 &   0 &   -1
\end{pmatrix},$$
where $\displaystyle S_{ij}=\sum_{a=1}^n\widetilde{g}(\widetilde{R}(e_i,e_a)e_a,e_j)$ for all $i,j=1,2,3,4$.
Thus, the scalar curvature of $(Nil^{4},\widetilde{g})$ is $\tau=-1$.
\end{corollary}


\section{Codazzi hypersurfaces in $Nil^{4}$}
Let $ (M^{3},g)$ be a hypersurface in $(Nil^{4},\widetilde{g})$. We have, $\xi=ae_{1}+be_{2}+ce_{3}+de_{4}$ the unit normal vector field on $( M^{3},g)$, where $a,b,c,d$ are local functions on $ M^{3} $. Thus
 \begin{eqnarray*}
 X_{1}=be_{1}-ae_{2} \,,&  &\, X_{2}=ce_{1}-ae_{3}\\
 X_{3}=de_{1}-ae_{4} \,,&  &\, X_{4}=ce_{2}-be_{3} \\
 X_{5}=de_{2}-be_{4} \,,&  &\, X_{6}=de_{3}-ce_{4}
\end{eqnarray*}
 are tangent vectors fields to the hypersurface $(M^{3},g)$. Now, assume that the hypersurface $(M^{3},g)$ is Codazzi, that is
\begin{equation}\label{eq3.1}
 (\nabla_{X}h)(Y,Z)= (\nabla_{Y}h)(X,Z),\quad\forall X,Y,Z \in\mathfrak{X}(M^3).
\end{equation}
Then it follows from the equation of Codazzi (\ref{eq1.9}) that
\begin{equation}\label{eq3.2}
\widetilde{g}(\widetilde{R}(X_{i},X_{j})X_{k},\xi) =0,\quad\forall i,j,k \in\lbrace 1,...,6\rbrace.
\end{equation}
By using the curvature components given in Proposition \ref{prop2}, we get the following
$$\widetilde{g}( \widetilde{R}(X_{1},X_{2})X_{3},\xi)=\frac{1}{4}abd(a-c)=0,$$
from the which we prove that $a=0$ or $b=0$ or $d=0$ or  $a=c$.\\
$ \bullet $ If $a=0$, we have the following equations
\begin{eqnarray*}
\widetilde{g}( \widetilde{R}(X_{1},X_{4})X_{1},\xi) &=&-\frac{1}{4}b^{3}c=0,  \\
\widetilde{g}( \widetilde{R}(X_{2},X_{4})X_{4},\xi) &=&\frac{1}{4}c^2b^2+\frac{1}{4}c^2=0,\\
\widetilde{g}( \widetilde{R}(X_{1},X_{5})X_{4},\xi) &=&\frac{1}{2}b^{3}d+\frac{1}{4}bc^2d=0.
\end{eqnarray*}
Thus $c=0$ and $bd=0$. So that, $\xi=e_2$ or $\xi=e_4$. Note that, in the case where $\xi=e_2$,
the Lie bracket $[e_4,e_3]=e_2$ is not tangent vector field on $M^3$ despite $e_2$ and $e_4$ are tangent vectors fields on $M^3$. So,
by Frobenius Theorem (see \cite{W}), this case is unacceptable. Then we have $\xi=e_4$.\\
$ \bullet $ If $b=0$, we obtain the equations
\begin{eqnarray*}
\widetilde{g}( \widetilde{R}(X_{1},X_{2})X_{1},\xi) &=&\frac{1}{4}a^2(a^2-c^2)=0,  \\
\widetilde{g}( \widetilde{R}(X_{1},X_{3})X_{1},\xi) &=&-\frac{1}{4}a^2d(3a+c)=0.
\end{eqnarray*}
For $a=0$, we get $c=0$. Thus $\xi=e_4$. For $a=\pm c$, we find that $ad=0$. Hence
$\xi=e_4$ or $\xi=\frac{1}{\sqrt{2}}(e_1\pm e_3)$. Note that, in the case where $\xi=\frac{1}{\sqrt{2}}(e_1\pm e_3)$,
the Lie bracket $[e_4,e_2]$ is tangent vector field on $M^3$ because $e_2$ and $e_4$ are tangent vectors fields on $M^3$. But $\widetilde{g}([e_4,e_2],\xi)=
\widetilde{g}(e_1,\xi)=\frac{1}{\sqrt{2}}\neq0$, we obtain a contradiction with the fact that $\xi$ in normal to $M^3$. Therefore, $\xi=e_4$.\\
$ \bullet $ If $d=0$, we have the equations
\begin{eqnarray*}
\widetilde{g}( \widetilde{R}(X_{1},X_{2})X_{2},\xi) &=&-\frac{1}{4}ab(a-c)^2=0,  \\
\widetilde{g}( \widetilde{R}(X_{1},X_{4})X_{1},\xi) &=&\frac{1}{4}b(a-c)(a^2+b^2+ac)=0,\\
\widetilde{g}( \widetilde{R}(X_{2},X_{4})X_{4},\xi) &=&-\frac{1}{4}c(a-c)(b^2+c^2+ac)=0,\\
\widetilde{g}( \widetilde{R}(X_{1},X_{2})X_{1},\xi) &=&\frac{1}{4}a(a-c)(a^2+b^2+ac)=0,\\
\widetilde{g}( \widetilde{R}(X_{6},X_{2})X_{3},\xi) &=&-a^2c^2-\frac{1}{4}ac(a^2-c^2)=0.
\end{eqnarray*}
Hence $a=c=0$. Thus, $\xi=e_2$. It is unacceptable, because in this case $e_3$ and $e_4$ are tangent vectors fields on $M^3$ but $[e_4,e_3]=e_2$ is not tangent vector field on $M^3$.\\
$\bullet$ If $ a=c $, we get the following equations
\begin{eqnarray*}
\widetilde{g}( \widetilde{R}(X_{6},X_{2})X_{3},\xi) &=&-c^2(c^2+\frac{1}{2}d^2)=0,  \\
\widetilde{g}( \widetilde{R}(X_{1},X_{5})X_{6},\xi) &=&\frac{1}{2}b^2(c^2-d^2)=0.
\end{eqnarray*}
we obtain $c=0$ and $bd=0$. Hence, $\xi=e_4$. Here, $\xi=e_2$ is unacceptable.\\

\begin{theorem}
A hypersurface $(M^{3},g)$ in the Lie group $(Nil^{4},\widetilde{g})$ is Codazzi if and only if the
unit normal vector field to $(M^{3},g)$ is $\xi=e_4$.
\end{theorem}

\begin{proof}
According to the previous calculations, it suffices to show that
\begin{equation}\label{eq3.3}
\widetilde{g}(\widetilde{R}(X,Y)Z,\xi) =(\nabla_Y h)(X,Z)-(\nabla_X h)(Y,Z)=0,
\end{equation}
for all $X,Y,Z\in \mathfrak{X}(M^3)$ for $\xi=e_4$. You can easily check the equations
$$\widetilde{g}(\widetilde{R}(e_i,e_j)e_k,e_4)=0,\quad i,j,k=1,...,3.$$
\end{proof}

\begin{remark}
(1) According to (\ref{eq1.13}), a Codazzi hypersurface $ (M ^{3},g)$ in $(Nil^{4},\widetilde{g})$ which is given by
   \begin{eqnarray*}
	f:  (M^{3},g)  &\longrightarrow& (Nil^{4},\widetilde{g}), \\
      (x,y,z)&\longmapsto& (x,y,z,t_{0})
\end{eqnarray*}
where $t_{0}\in\mathbb{R}$. \\
(2) Since $B(X,Y)=(\widetilde{\nabla}_XY)^\perp$ for all $X,Y\in \mathfrak{X}(M^3)$. By using the Proposition \ref{prop1},
the second fundamental form $h$ of the hypersurface $ (M^{3},g) $ in $(Nil^{4},\widetilde{g})$ is given by
$$
 (h_{ij})=\begin{pmatrix}
 0 & \frac{1}{2}  & 0\\
 \frac{1}{2}  &  0 & \frac{1}{2} \\
 0 &\frac{1}{2} & 0
\end{pmatrix},
$$
where $h_{ij}=h(e_i,e_j)$ for all $i,j=1,...,3$. In this case, $(M^{3},g)$ is not totally geodesic because $h\neq 0$, and minimal hypersurface in $(Nil^{4},\widetilde{g})$, that is $H=0$. It is also parallel in $ (Nil^{4},\widetilde{g})$ because $\nabla h=0$.\\
(3) The principal curvatures of $(M^{3},g)$ in $ (Nil^{4},\widetilde{g})$ are $-\frac{1}{\sqrt{2}},0,\frac{1}{\sqrt{2}}$.

One can compute the matrix of Ricci curvature of $ (M^{3},g) $ and get
$$(S_{ij})=\begin{pmatrix}
 \frac{1}{4} & 0 & \frac{1}{4}\\
        0    & \frac{1}{2} & 0\\
        0    & 0 & \frac{1}{4}
\end{pmatrix}.$$
So the scalar curvature of $ (M^{3},g) $ will be by
 $$\tau=1.$$

From the second fundamental form we can get the matrix
of shape operator of $ (M^{3},g) $ in $(Nil^{4},\widetilde{g})$ as so
$$ A_{\xi}=\begin{pmatrix}
 0 & \frac{1}{2} & 0\\
        \frac{1}{2}   & 0 & \frac{1}{2}\\
        0    & \frac{1}{2} & 0
\end{pmatrix}.$$

\end{remark}
From the previous remarks we can conclude the following Corollary.
\begin{corollary}
A hypersurface $(M^{3},g)$ in the Lie group $(Nil^{4},\widetilde{g})$ is Codazzi if and only if it is parallel.
The Lie group $(Nil^{4},\widetilde{g})$ do not have any totally geodesic hypersurface.
\end{corollary}


\section{Minimal hypersurfaces in $Nil^{4}$}

Let $ (M^{3},g)$ be a hypersurface in $(Nil^{4},\widetilde{g})$. In this section, we search the conditions for the hypersurface $(M^{3},g)$ to be minimal in $(Nil^{4},\widetilde{g})$, where the unit normal vector field on $( M^{3},g)$ is given by $\xi=ae_{1}+be_{2}+ce_{3}+de_{4}$,
and assume that $\{X_i\}_{1\leq i\leq 3}$ is a local orthonormal frame on $(M^{3},g)$, where
$X_i=a_ie_{1}+b_ie_{2}+c_ie_{3}+d_ie_{4}$ for some local functions $\{a,b,c,d,a_i,b_i,c_i,d_i\}_{1\leq i\leq 3}$ on $M^3$ depends only on the variable $t$.

\begin{theorem}\label{th4}
The hypersurface $(M^{3},g)$ is minimal in $(Nil^{4},\widetilde{g})$ if and only if
$$\sum_{i=1}^3\left[a_i\left(b_id-bd_i\right)+b_i\left(c_id-cd_i\right)
+d_i\left(aa'_i+bb'_i+cc'_i+dd'_i\right)\right]=0.$$
\end{theorem}

\begin{proof}
Let $i=1,2,3$. We compute
\begin{eqnarray}\label{eq4.1}
 \widetilde{\nabla}_{X_i}X_i
   &=&\nonumber \widetilde{\nabla}_{a_ie_{1}+b_ie_{2}+c_ie_{3}+d_ie_{4}}\left(a_ie_{1}+b_ie_{2}+c_ie_{3}+d_ie_{4}\right)  \\
   &=&\nonumber a_i\left( a_i\widetilde{\nabla}_{e_1}e_1+b_i\widetilde{\nabla}_{e_1}e_2+c_i\widetilde{\nabla}_{e_1}e_3+d_i\widetilde{\nabla}_{e_1}e_4\right)\\
   &&\nonumber +b_i\left( a_i\widetilde{\nabla}_{e_2}e_1+b_i\widetilde{\nabla}_{e_2}e_2+c_i\widetilde{\nabla}_{e_2}e_3+d_i\widetilde{\nabla}_{e_2}e_4\right)\\
   &&\nonumber +c_i\left( a_i\widetilde{\nabla}_{e_3}e_1+b_i\widetilde{\nabla}_{e_3}e_2+c_i\widetilde{\nabla}_{e_3}e_3+d_i\widetilde{\nabla}_{e_3}e_4\right)\\
   &&\nonumber +d_i\Big(a'_ie_{1}+a_i\widetilde{\nabla}_{e_4}e_1
   +b'_ie_2+b_i\widetilde{\nabla}_{e_4}e_2
   +c'_ie_3+c_i\widetilde{\nabla}_{e_4}e_3\\
   &&+d'_ie_{4}+d_i\widetilde{\nabla}_{e_4}e_4\Big).
\end{eqnarray}
From Proposition \ref{prop1}, and equation (\ref{eq4.1}), we obtain
\begin{eqnarray*}
 \widetilde{\nabla}_{X_i}X_i
   &=&
   a_i\left(\frac{b_i}{2}e_4-\frac{d_i}{2}e_2\right)
   +b_i\left(\frac{a_i}{2}e_4+\frac{c_i}{2}e_4-\frac{d_i}{2}(e_1+e_3)\right)\\
   &&+c_i\left(\frac{b_i}{2}e_4-\frac{d_i}{2}e_2\right)
   +d_i\Big(a'_ie_{1}-\frac{a_i}{2}e_2
   +b'_ie_2+\frac{b_i}{2}(e_1-e_3)\\
   &&+c'_ie_3+\frac{c_i}{2}e_2+d'_ie_{4}\Big),
\end{eqnarray*}
it is equivalent to the following equation
\begin{equation}\label{eq4.2}
\widetilde{\nabla}_{X_i}X_i=a'_id_ie_{1}+d_i(b'_i-a_i)e_2+d_i(c'_i-b_i)e_3+[d_id'_i+b_i(a_i+c_i)]e_4.
\end{equation}
By equation (\ref{eq4.2}), we have
\begin{eqnarray}\label{eq4.3}
 \widetilde{g}(\widetilde{\nabla}_{X_i}X_i,\xi)
   &=&\nonumber aa'_id_i+bd_i(b'_i-a_i)+cd_i(c'_i-b_i)+d[d_id'_i+b_i(a_i+c_i)].\\
\end{eqnarray}
Note that, $B(X_i,X_i)=(\widetilde{\nabla}_{X_i}X_i)^\perp$, that is $h(X_i,X_i)=\widetilde{g}(\widetilde{\nabla}_{X_i}X_i,\xi)$. Thus, the hypersurface $(M^3,g)$ is minimal if
\begin{equation}\label{eq4.4}
H=\frac{1}{3}\sum_{i=1}^3\widetilde{g}(\widetilde{\nabla}_{X_i}X_i,\xi)=0.
\end{equation}
The Theorem \ref{th4} follows by  equations  (\ref{eq4.3}) and (\ref{eq4.4}).
\end{proof}

\begin{example}
We consider the following vector fields
\begin{eqnarray*}
   \xi&=& \frac{2}{\sqrt{5}(2+t^2)}e_1+ \frac{2 t}{\sqrt{5}(2+t^2)}e_2+\frac{t^2}{\sqrt{5}(2+t^2)}e_3-\frac{2}{\sqrt{5}}e_4,\\
   X_1&=& -\frac{t}{\sqrt{1+t^2}} e_1+\frac{1}{\sqrt{1+t^2}} e_2,\\
   X_2&=& -\frac{t^2}{(2+t^2)\sqrt{1+t^2}}e_1-\frac{t^3}{(2+t^2)\sqrt{1+t^2}}e_2+\frac{2\sqrt{1+t^2}}{2+t^2}e_3,\\
   X_3&=& \frac{4}{\sqrt{5}(2+t^2)}e_1+\frac{4t}{\sqrt{5}(2+t^2)}e_2+\frac{2 t^2}{\sqrt{5}(2+t^2)}e_3+\frac{1}{\sqrt{5}}e_4.
\end{eqnarray*}
It is easy to verify that these vector fields satisfy $$\widetilde{g}(\xi,\xi)=1,\quad\widetilde{g}(\xi,X_i)=0,\quad
\widetilde{g}(X_i,X_j)=\delta_{ij},\quad\forall i,j=1,2,3,$$ and the condition of Theorem \ref{th4}.
Thus the hypersurface $(M^{3},g)$ defined by these vector fields is minimal.
According to (\ref{eq1.13}), this hypersurface $(M^{3},g)$ is given by
\begin{eqnarray*}
	f:  (M^{3},g)  &\longrightarrow& (Nil^{4},\widetilde{g}). \\
      (y,z,t)&\longmapsto& (2t+\frac{t^3}{3},y,z,t)
\end{eqnarray*}
\end{example}



\normalsize\section*{\large Acknowledgements}\vspace{-0.3cm}
\noindent
The authors are supported by National Agency Scientific Research of Algeria.

\vspace{-0.6cm}








\end{document}